\documentclass[11pt,a4paper, twoside]{amsart}
\usepackage{url}
\usepackage{amssymb}
\usepackage{color}
\usepackage{fancyhdr}
\usepackage{colonequals}
\usepackage{graphicx}
\usepackage{pinlabel}
\usepackage{mathtools} 
\usepackage[all]{xy}
\usepackage{pinlabel}
\usepackage{tikz-cd}
\usepackage{epsfig}
\usepackage{epstopdf}

\usetikzlibrary{arrows,shapes}

\tikzset{join/.code=\tikzset{after node path={%
\ifx\tikzchainprevious\pgfutil@empty\else(\tikzchainprevious)%
edge[every join]#1(\tikzchaincurrent)\fi}}}

\tikzset{>=stealth',every on chain/.append style={join},
         every join/.style={->}}

\usepackage{pinlabel,cite}
\input{epsf.tex}

 \usepackage{labelfig}
 \input{epsf.tex}

\renewcommand{\epsilon}{\varepsilon}
\renewcommand{\setminus}{\smallsetminus}
\renewcommand{\emptyset}{\varnothing}

\newtheorem*{namedtheorem}{\theoremname}
\newcommand{\theoremname}{testing}

\newtheorem{theorem}{Theorem}[section]
\newtheorem{proposition}[theorem]{Proposition}
\newtheorem{corollary}[theorem]{Corollary}
\newtheorem{lemma}[theorem]{Lemma}

\newtheorem{question}[theorem]{Question}

\theoremstyle{definition}
\newtheorem{example}[theorem]{Example}

\newtheorem{definition}[theorem]{Definition}

\theoremstyle{remark}
\newtheorem{remark}[theorem]{Remark}

\newcommand{\Z}{\mathbb Z}
\newcommand{\N}{\mathbb N}

\newcommand{\CA}{\mathcal{A}}

\newcommand{\GL}{\operatorname{GL}}

\newcommand{\PSL}{\operatorname{PSL}}

\newcommand{\cohom}[3]{H^{{\raise1pt\hbox{$\scriptstyle#1$}}}(#2\>\!,#3)}
\newcommand{\tatecohom}[3]%
  {\widehat H^{{\raise1pt\hbox{$\scriptstyle#1$}}}(#2\>\!,#3)}

\newcommand{\Cohom}[3]%
  {H^{{\raise1pt\hbox{$\scriptstyle#1$}}}\big(#2\>\!,#3\big)}
\newcommand{\Tatecohom}[3]%
  {\widehat H^{{\raise1pt\hbox{$\scriptstyle#1$}}}\big(#2\>\!,#3\big)}

\newcommand{\homol}[3]{H_{{\lower1pt\hbox{$\scriptstyle#1$}}}(#2\>\!,#3)}
\newcommand{\homolog}[2]{H_{{\lower1pt\hbox{$\scriptstyle#1$}}}(#2)}

\newcommand{\im}{\operatorname{Im}}

\renewcommand{\implies}{\Rightarrow}

\newcommand{\Isom}{\operatorname{Isom}}

\DeclareMathOperator{\Out}{Out}

\DeclareMathOperator{\Aut}{Aut}
\DeclareMathOperator{\PAut}{PAut}

\DeclareMathOperator{\Inn}{Inn}

\newcommand{\lk}{\operatorname{lk}}
\newcommand{\st}{\operatorname{st}}
\newcommand{\slk}{\operatorname{slk}}
\newcommand{\sst}{\operatorname{sst}}

\begin{document}

\title[Representations of $\PAut(A_\Gamma)$]{Representations of pure symmetric automorphism groups of RAAGs}

\author{Javier Aramayona \& Conchita Mart\'inez P\'erez}

\date{\today}

\thanks{The first named author is partially supported by RYC-2013-13008 and the second one by Gobierno de Arag\'on and European Regional 
Development Funds. Both authors were funded by grant MTM2015-67781}

\maketitle

\begin{abstract}
We study representations of the pure symmetric automorphism group $\PAut(A_\Gamma)$ of a RAAG $A_\Gamma$  with defining graph $\Gamma$. 

We first construct a homomorphism from $\PAut(A_\Gamma)$ to the direct product of  a RAAG and a finite direct product of copies of $F_2 \times F_2$; moreover, the image of $\PAut(A_\Gamma)$ under this homomorphism  is surjective onto each factor. As a consequence, we obtain interesting actions of $\PAut(A_\Gamma)$ on non-positively curved spaces

We then exhibit, for connected $\Gamma$, a RAAG which property contains $\Inn(A_\Gamma)$ and embeds as a normal subgroup of $\PAut(A_\Gamma)$. We end with a discussion of the linearity problem for $\PAut(A_\Gamma)$. \end{abstract}

\section{Introduction}
 Given a  group $G$ with a finite generating set $X$, the {\em pure symmetric automorphism group} $\PAut(G)$ is the group of those automorphisms of $G$ that  send every generator in $X$ to a conjugate of itself. In this paper, we will be interested in representations of pure symmetric automorphism groups of {\em right-angled Artin groups} (RAAGs) with respect to the standard generating set;
 see section \ref{sec:defs} for definitions.

An important and motivating example is that of the free group $F_n$ on $n$ letters $x_1,\ldots,x_n$. Humphries \cite{Hum} proved that $\PAut(F_n)$ is generated by the  {\em partial conjugations} $x_j \mapsto x_ix_jx_i^{-1}$. This revealed a strong relation between $\PAut(F_n)$ and braid groups, for Goldsmith \cite{Goldsmith} had previously showed that the subgroup of $\Aut(F_n)$ generated by partial conjugations is isomorphic to the fundamental group of the configuration space of $n$ unknotted, unlinked circles in $\mathbb{S}^3$. Shortly afterwards, McCool  \cite{McCool} gave an explicit presentation of $\PAut(F_n)$ in terms of partial conjugations, where every relation is a commutativity relation between (products of) generators. 

For a general RAAG $A_\Gamma$, Humpries's result above was extended by Laurence \cite{Laurence2}, who proved that $\PAut(A_\Gamma)$ is generated by partial conjugations. In addition, Toinet \cite{Toinet} proved that  $\PAut(A_\Gamma)$ has a presentation similar in spirit to McCool's presentation of $\PAut(F_n)$. This was subsequently refined by Koban-Piggott \cite{KP}, who showed that it suffices to consider a finite set of commutation relations between (products of) partial conjugations of $A_\Gamma$.

\subsection{Subdirect images of $\PAut(A_\Gamma)$} 
The first aim of this paper is to use Koban-Piggott's presentation \cite{KP} of $\PAut(A_\Gamma)$ in order to produce a (non-injective) homomorphism of $\PAut(A_\Gamma)$ to the product of a  RAAG times a direct product of copies of $F_2 \times F_2$. Recall that a subgroup of a direct product of groups is said to be a {\em subdirect product} if its projection onto each direct factor is surjective. We will show:

\begin{theorem}
For every graph $\Gamma$ there exist a graph $\Delta$, a number $N\ge 0$, and a homomorphism  \[ \PAut(A_\Gamma) \to A_{\Delta} \times \prod_{i=1}^N (F_2 \times F_2)\] whose image is a subdirect product. Moreover, the image of every partial conjugation has infinite order. 
\label{thm:repraag}
\end{theorem}

In fact, the homomomorphism of Theorem \ref{thm:repraag} is constructed explicitly,  and $N$ and $\Delta$ are determined in terms of the graph $\Gamma$. In this direction, following \cite{Charney} we say that a pair of vertices $v,w$ of $\Gamma$ form a {\em SIL-pair}  if $v \notin \lk(w)$ and $\Gamma \setminus \lk(v)\cap \lk(w)$ has a connected component $Y$ with $v,w \notin Y$.
 Then the number $N$ in Theorem \ref{thm:repraag} is equal to  half the number of {\em SIL-type} relations in Koban-Piggott's presentation of $\PAut(A_\Gamma)$ (see the comment after the proof of Theorem  \ref{thm:repraag}), and $A_\Delta$ is the RAAG obtained from $\PAut(A_\Gamma)$ by killing every partial conjugation that  
 appears  in a SIL-type relation. In the specific case when $A_\Gamma=F_n$, every partial conjugation appears in a SIL-type relation, and thus we have $\Delta=\emptyset$ and $N=\frac{n(n-1)(n-2)}{2}$.

It is easy to see that the homomorphism given by Theorem \ref{thm:repraag} is not injective unless $\PAut(A_\Gamma)$ is a RAAG itself; see Proposition \ref{lem:nonfaithful}. Hence a natural question is:

\begin{question}
Determine the kernel of the homomorphism in Theorem \ref{thm:repraag}. 
\end{question}

\noindent{\bf Applications to  ${\rm CAT}(0)$-actions.} 
A result of Bridson \cite[Theorem 1.1]{Bridson} asserts that, for $n$ sufficiently large, if $\Aut(F_n)$ acts by (semisimple) isometries on a complete ${\rm CAT}(0)$ space, then every transvection (and thus every partial conjugation also) is elliptic, i.e. it fixes a point. By considering induced actions, the same result holds for finite-index subgroups of $\Aut(F_n)$ too. 

However, the presence of transvections is crucial for these results. Indeed, Koban-Piggott's presentation of $\PAut(F_n)$ \cite{KP} implies that $\PAut(F_n)$ (which recall is generated by all partial conjugations) has non-trivial abelianization, and thus acts by hyperbolic isometries on a ${\rm CAT}(0)$ space $X$. However, the action so obtained is {\em abelian}, in the sense that the image of the homomorphism $\PAut(F_n) \to \Isom(X)$ is abelian. 

In sharp contrast, composing the homomorphism $\rho$ from Theorem \ref{thm:repraag} with a faithful representation of $F_2$ into (say) $\PSL(2,\Z)$, we obtain the following:

\begin{corollary}
For every $n>0$, $\PAut(F_n)$ admits a non-abelian ${\rm CAT}(0)$-action where every partial conjugation acts hyperbolically.
\label{cor:hyp}
\end{corollary}

\subsection{Normal RAAG subgroups of $\PAut(A_\Gamma)$} A well-known corollary of a result of Davis-Januszkiewicz \cite{DavisJanus} is that RAAGs are linear. As a consequence, we may use the homomorphism constructed in Theorem \ref{thm:repraag} to construct a non-faithful linear representation of $\PAut(A_\Gamma)$. In light of the analogy between $\PAut(F_n)$ and braid groups alluded to above, a natural (open) question is whether $\PAut(A_\Gamma)$ is linear. It is easy to see that the group of {\em inner automorphisms} $\Inn(A_\Gamma)$ of $A_\Gamma$, which is isomorphic to the quotient of $A_\Gamma$ by its center, is itself a RAAG that embeds as a normal subgroup of $\PAut(A_\Gamma)$.

 In the second part we construct, for connected $\Gamma$, a graph $\hat{\Gamma}$ such that $A_{\hat{\Gamma}}$ properly contains $\Inn(A_\Gamma)$, and it still embeds in $\PAut(A_\Gamma)$ thus providing a big linear normal subgroup of $\PAut(A_\Gamma)$. More concretely, we will show:

\begin{theorem} Let $\Gamma$ be a connected graph. Then there exists a right-angled Artin group $A_{\hat \Gamma}$ and an injective map  
$$A_{\hat \Gamma} \to\PAut(A_\Gamma)$$
whose image is normal and contains $\Inn(A_\Gamma)$. Moreover if $\Gamma$ has no SIL pair, this map is an isomorphism.
\label{thm:main2}
\end{theorem}

In \cite[Section 3]{ChV2}, Charney and Vogtmann consider a subgroup $K$ of the {\em outer} automorphism group of $A_\Gamma$ which is generated by cosets of certain partial conjugations, and which they prove is abelian. As it turns out, a similar reasoning to that of the proof of Theorem \ref{thm:main2} implies that the lift $\tilde K$ of $K$ to $\Aut(A_\Gamma)$ is a RAAG, and that it embeds as a normal subgroup of $\PAut(A_\Gamma)$. However, in general  $\tilde K$ will be properly contained in the group $A_{\hat \Gamma}$ given by Theorem \ref{thm:main2}. See Remark \ref{rem:CV} for more details.  

\medskip

The plan of the paper is as follows. In section \ref{sec:defs} we will give the necessary background on right-angled Artin  groups and their automorphism groups.  Then in section
\ref{sec:raags} we proceed to construct the homomorphism of Theorem \ref{thm:repraag}. Finally, in section \ref{sec:conch} we prove Theorem \ref{thm:main2}.


\medskip

\noindent{\bf Acknowledgements.} These ideas sprouted from conversations during the $11^{ th}$ {\em Barcelona Weekend in Group Theory} in May 2016. The authors are grateful to the organizer of the conference, Pep Burillo.

\section{Graphs, RAAGs, and their automorphisms}
\label{sec:defs}

In this section we will briefly review the necessary definitions needed for the proofs of our main results. 

\subsection{Graphs}\label{graphs}  Throughout, $\Gamma$ will denote a finite simplicial graph. We will write $V(\Gamma)$ and $E(\Gamma)$, respectively, for the set of vertices and edges of $\Gamma$. 

Given a vertex $v\in V(\Gamma)$, the {\em link} $\lk_\Gamma(v)$ of $v$ in $\Gamma$ is the set of those vertices of $\Gamma$ which are connected by an edge to $v$. The {\em star} of $v$ is $st_\Gamma(v) := \lk_\Gamma(v) \cup \{v\}$. If there is no risk of confusion, we will simply write $\lk(v)$ and $\st(v)$ in order to relax notation. 

As mentioned in the introduction, we say that the vertices $v,w$ {\em form a SIL} if $v \notin \lk(w)$ and $\Gamma \setminus \lk(v)\cap \lk(w)$ has a connected component which does not contain either $v$ or $w$. 

Suppose that $v,w$ are not linked. The following terminology comes from \cite{DW} (see also \cite{GPR}). A connected component of $\Gamma - \st(v)$ that is also a connected component of $\Gamma - \st(w)$ will be called a {\em shared} component. The unique component of $\Gamma - \st(v)$ that contains $w$ will be called the {\em dominant} component of $v$. Finally, a {\em subordinate component} of $w$ is any connected component of $\Gamma - \st(w)$ that is contained in the dominant component of $v$. By \cite[Lemma 2.1]{DW}, any connected component of $\Gamma - \st(v)$ is of one of these types; moreover $v$ and $w$ form a SIL if and only if there is some shared component.

\subsection{Right-angled Artin  groups} Let $\Gamma$ be a finite simplicial graph. The {\em right-angled Artin group} (RAAG, for short) defined by $\Gamma$ is the group with presentation \[A_\Gamma= \langle v \in V(\Gamma) \mid [v,w] = 1 \iff vw \in E(\Gamma)\rangle . \]

As usual, in order to keep the notation under control,  we are blurring the distinction between vertices of $\Gamma$ and the corresponding generators of $A_\Gamma$. 

Notable examples are the extreme cases when $\Gamma$ has no edges, so that $A_\Gamma$ is a free group, and when $\Gamma$ is a complete graph, in which case $A_\Gamma$ is a free abelian group.

\subsection{Automorphisms of RAAGs} 
Let $A_\Gamma$ the RAAG defined by the graph $\Gamma$, and consider its automorphism group $\Aut(A_\Gamma)$. In the light of the precedent paragraph, for a fixed number $n$ of vertices of $\Gamma$, the group $\Aut(A_\Gamma)$ interpolates between the cases of $\Aut(F_n)$ and $\GL(n,\Z) = \Aut(\Z^n)$.

Laurence \cite{Laurence2} and Servatius \cite{Servatius} gave an explicit finite set of generators for $\Aut(A_\Gamma)$ for arbitrary $\Gamma$. We now briefly recall these generators:

\begin{enumerate}
\item {\em Graphic automorphisms.} These are the elements of $\Aut(A_\Gamma)$ induced by the symmetries of $\Gamma$. 
\item {\em Inversions.} Given $v \in V(\Gamma)$, the inversion on $v$ is the automorphism that maps $v$ to $v^{-1}$ and fixes the rest of generators of $A_\Gamma$. 
\item {\em Transvections.} Given vertices $v,w\in V(\Gamma)$ with $\lk(v) \subset \st(w)$, the transvection $t_{vw}$ is the automorphism of $A_\Gamma$ given by
$$\Bigg\{\begin{aligned}
t_{vw}(v)&=vw, \\
t_{vw}(z)&=z,\, z\neq v.\\
\end{aligned}$$
\item {\em Partial conjugations.} Given $v\in V(\Gamma)$ and a connected component $A$ of $\Gamma - \st(v)$, the partial conjugation of $A$ by $v$ is the automorphism $c_{A,v}$ of $A_\Gamma$ given by $$\Bigg\{\begin{aligned}
c_{A,v}(w)&=v^{-1}wv, w\in A\\
c_{A,v}(z)&=z,\, z\notin A.\\
\end{aligned}$$
\end{enumerate}

Laurence \cite{Laurence2} and Servatius \cite{Servatius} proved: 

\begin{theorem}[\cite{Laurence2,Servatius}]
$\Aut(A_\Gamma)$ is generated by the four types of automorphisms described above. In particular, it is finitely generated.
\label{thm:LS}
\end{theorem}

For the sake of completeness, although it will not be needed here, we remark that Day \cite{Day} subsequently proved that $\Aut(A_\Gamma)$ is finitely presented, and gave an explicit presentation.

\subsection{Pure symmetric automorphism groups}
We will consider the subgroup $\PAut(A_\Gamma)$ of $\Aut(A_\Gamma)$ consisting of those automorphisms that send every generator of $A_\Gamma$ to a conjugate of itself.  Laurence proved: 

\begin{theorem}[Laurence \cite{Laurence2}]
$\PAut(A_\Gamma)$ coincides with the subgroup of $\Aut(A_\Gamma)$ generated by partial conjugations. 
\end{theorem}

As mentioned in the introduction, the above theorem was proved earlier by Humphries \cite{Hum} in the particular case when $A_\Gamma$ is a free group. Observe that a
consequence of Theorem \ref{thm:LS} above is that if $\Gamma$ does not have vertices $v,w$ with $\lk(v) \subset \st(w)$, then $\PAut(A_\Gamma)$ has finite index in $\Aut(A_\Gamma)$, since in this case there are no transvections. 

As mentioned in the introduction, McCool \cite{McCool} gave an explicit presentation of $\PAut(F_n)$. We record (an equivalent form of) this presentation next. Let $v_1,\ldots, v_n$ the standard basis of $F_n$, and denote by $c_{ij}$ the automorphism of $F_n$ given by conjugating $v_i$ by $v_j$. We have: 

\begin{theorem}[McCool \cite{McCool}]
$\PAut(F_n)$ is generated by the set $\{c_{ij} \mid 1\le i,j \le n\}$, and a complete set of relations is: 
\begin{enumerate}
\item[(i)] $[c_{ij}, c_{kl}]=1 $, whenever $\{i,j\} \cap \{k,l\} =\emptyset$ or $j=l$, and
\item[(ii)] $[c_{ij}c_{kj}, c_{ik}]=1$ whenever $i,j,k$ are pairwise distinct. 
\end{enumerate}
\label{thm:mccool}
\end{theorem}

Inspired by the above theorem, and using Day's presentation of $\Aut(A_\Gamma)$ \cite{Day}, Toinet \cite{Toinet} computed an explicit finite presentation $\PAut(A_\Gamma)$. This presentation was subsequently refined by Koban-Piggott \cite{KP}; we state their result here, slightly reformulated to suit our purposes:  


\begin{theorem}[Koban-Piggott \cite{KP}]\label{presA}
$\PAut(A_\Gamma)$ is generated by all partial conjugations $c_{A,v}$, subject to the following relations: 
\begin{itemize}
\item[(i)] $[c_{A,v},c_{B,w}]=1$ if either $v=w$ or $v \in \lk(w)$,

\item[(ii)] $[c_{A,v},c_{B,w}]=1$ if $v\neq w$ and either $(A\cup\{v\})\bigcap(B\cup\{w\}) = \emptyset$ or $(A\cup\{v\})\subseteq B$ or $(B\cup\{w\})\subseteq A$,

\item[(iii)] $[c_{A,v}c_{B,v},c_{A,w}]=1$ whenever $\Gamma-\lk(v)\cap\lk(w)$ has a connected component $A$ which contains neither $v$ nor $w$, and $w\in B$. 
\end{itemize}
 \label{thm:prespaut}
\end{theorem}

As in \cite{DW}, it is convenient to reformulate these reations using the terminology introduced at the end of Subsection \ref{graphs}. So we have that the partial conjugations $c_{A,v}$ and $c_{B,w}$ commute if and only if either at least one of the connected components $A$ and $B$ is subordinate or both are shared but distinct. 
And case (iii) happens if $A$ is shared and $B$ dominant (and in this case, the pair of vertices $v,w \in V(\Gamma)$ must form a SIL). We will thus refer to relations of type (iii) as {\em SIL-type} relations in $\PAut(A_\Gamma)$.

\begin{remark}
Observe that, in the case of $A_\Gamma=F_n$,  the presentations of $\PAut(F_n)$ given by Theorems \ref{thm:mccool} and \ref{thm:prespaut} coincide. In the case when $A_\Gamma=F_n$, any pair of vertices $v_j$ and $v_k$ form a SIL (for $n\ge 3$) with shared components of the form $\{v_i\}$ for any other $i$. 
\end{remark}

\subsection{Arrow diagrams} Next, we present a useful combinatorial manner of visualizing $\PAut(A_\Gamma)$, in terms of {\em arrow diagrams}, as we now introduce.

Given  a finite simplicial graph $\Lambda$,  as before we denote its  vertex set by $V(\Lambda)$ and its edge set by $E(\Lambda)$.

\begin{definition}
An {\em arrow diagram} is a pair $(\Lambda, \CA)$, where $\Lambda$ is a finite simplicial graph and $\CA$ is a  subset of $V(\Lambda) \times E(\Lambda)$. The elements of $\CA$ are called {\em arrows}. Given an arrow $(v,e) \in \CA$, we say that $v$ is  its {\em initial vertex}. Finally, we call $\Lambda$ the {\em graph underlying the arrow diagram}. 
\end{definition}

The motivation for the name ``arrow diagram" is that we may represent an arrow diagram as a graph together with some arrows, see figure   \ref{fig:arrows}.

\begin{figure}[h]
\leavevmode \SetLabels
\L(.28*.25) $c_{31}$\\
\L(.28*.75) $c_{21}$\\
\L(.44*.25) $c_{12}$\\
\L(.44*.76) $c_{32}$\\
\L(.66*.25) $c_{23}$\\
\L(.66*.75) $c_{13}$\\
\L(.39*.31) $\alpha_1$\\
\L(.39*.54) $\alpha_2$\\
\L(.56*.54) $\alpha_4$\\
\L(.56*.31) $\alpha_3$\\
\L(.45*.04) $\alpha_5$\\
\L(.52*.92) $\alpha_6$\\
\endSetLabels
\begin{center}
\AffixLabels{\centerline{\epsfig{file=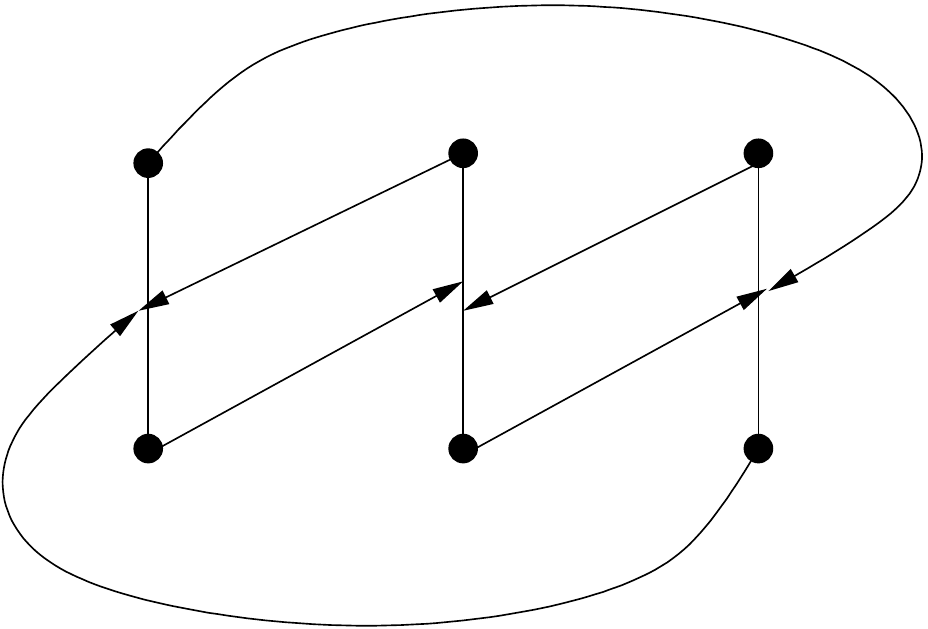,width=6.0cm,height=4cm, angle=0} }}
\vspace{-10pt}
\end{center}
\caption{The graphical representation of an arrow diagram. In fact, the reader may check from Theorem \ref{thm:mccool} that this is the arrow diagram of $\PAut(F_3)$
.} \label{fig:arrows}
\end{figure}
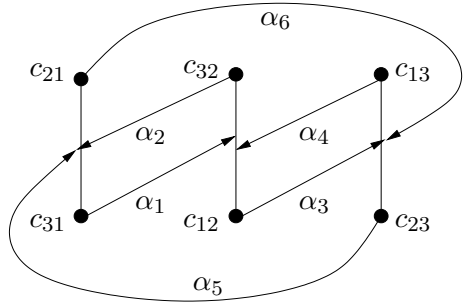

We are now going to define a group from an arrow diagram. We choose to do this is full generality, hoping that the class of groups so defined may be of independent interest. 

\begin{definition}
Let $(\Lambda,\CA)$ be an arrow diagram. The {\em arrow group} $G(\Lambda, \CA)$ defined by $(\Lambda, \CA)$ is the group generated by the vertices of $\Lambda$, subject to the following relations: 
\begin{enumerate}
\item $[v,w] = 1$ if and only if $v \in \lk(w)$.
\item $[u, vw] =1$ if and only if $v, w$ span an edge $e$ and $(u,e)\in \CA$. 
\end{enumerate}
\label{def:arrowgroup}
\end{definition}

Observe that if $\CA = \emptyset$, then $G(\Lambda, \CA)=A_\Lambda$, the right-angled Artin group defined by $\Lambda$. 

In the light of Theorem \ref{thm:prespaut}, the following result is immediate: 

\begin{proposition}
Let $\Gamma$ be a finite simplicial graph. Then $\PAut(A_\Gamma)$ is an arrow group. 
\end{proposition}

Note that the graph underlying the arrow diagram of $\PAut(A_\Gamma)$ is, in general, not isomorphic to $\Gamma$.

\subsubsection{Dual arrows} Viewed as an arrow group, $\PAut(A_\Gamma)$ has a further interesting property, which plays an important r\^ole in the proof of Theorem \ref{thm:repraag}. More concretely, we say that  an arrow diagram $(\Lambda, \CA)$ has {\em property (D)} if for every arrow $\alpha=(v,e)\in \CA$, there exists a unique arrow $\bar{\alpha}= (\bar{v}, \bar{e})\in \CA$ such that $\bar{v}$ is a vertex of $e$, $v$ is a vertex of $\bar{e}$, and the subgraph of $\Lambda$ spanned by the vertices of $e,\bar{e}$ has no other edges than $e$ and $\bar{e}$. 
In this situation, we say that $\alpha$ and $\bar{\alpha}$ are {\em dual arrows}. For instance, the arrows $\alpha_1$ and $\alpha_2$ of Figure \ref{fig:arrows} are dual to each other.  We deduce: 

\begin{proposition}
Let $\Gamma$ be a finite simplicial graph. Then the arrow diagram of $\PAut(A_\Gamma)$ has property (D).
\label{prop:pconj}
\end{proposition}

\begin{proof}
Consider an arrow $(v,e) \in \CA$ in $\PAut(A_\Gamma)$. By Theorem \ref{thm:prespaut}, it corresponds to a SIL-type relation \[[c_{A,v}c_{B,v},c_{A,w}]=1,\] for $v,w$ a SIL-pair, where $A$ is a shared component and $B$ is dominant thus  $w \in B$. Let $C$ be the other dominant component, i.e., the  connected component of $\Gamma - \lk(w)$ that contains $v$. Then the partial conjugation $c_{C,w}$ is well-defined and we have \[[c_{A,w}c_{C,w}, c_{A,v}] =1,\] as desired. Moreover, the only edges between $c_{A,w}$, $c_{C,w}$, $c_{A,v}$ and $c_{B,v}$ are the edges corresponding to $[c_{A,w},c_{C,w}]=1$ and $[c_{A,v},c_{B,w}]=1$.
Thus property (D) holds. 
\end{proof}

\begin{corollary} Let $\Gamma$ be a finite simplicial graph. Then the arrow diagram of $\PAut(A_\Gamma)$ has an even number of arrows. 
\label{cor:even}
\end{corollary}

\begin{remark}
 Theorem \ref{thm:mccool} implies that the number of arrow pairs in the arrow diagram for $\PAut(F_n)$ is equal to \[ \frac{n(n-1)(n-2)}{2}.\] Indeed, whenever $\{i,j,k\}$ are pairwise distinct, there is an arrow $(c_{ik},e(c_{ij},c_{kj}))$, where $e(c_{ij},c_{kj})$ is the edge between $c_{ij}$ and $c_{kj}$. This arrow has dual arrow $(c_{ij},e(c_{jk},c_{ik}))$. Therefore, for each choice of $i\neq j$ we have $(n-2)/2$ pairs.
\end{remark}

\section{Proof of Theorem \ref{thm:repraag}}
\label{sec:raags}

Before proceeding to proving Theorem \ref{thm:repraag}, we need a final piece of notation. Let  $\Gamma$ be a finite simplicial graph and $(\Lambda, \CA)$ the arrow diagram of $\PAut(A_\Gamma)$. We say that a vertex $w\in V(\Gamma)$ is {\em not involved in any arrows} if 
the following two conditions hold
\begin{itemize}
\item[i)] there is no arrow of the form $(w,e)$,

\item[ii)] there is no arrow of the form $(v,e)$ where $w$ is a vertex of $e$.
\end{itemize}
\begin{proof}[Proof of Theorem \ref{thm:repraag}]
As above, denote by $(\Lambda, \CA)$ the arrow diagram of $\PAut(A_\Gamma)$. In the light of Corollary \ref{cor:even}, we have that $|\CA| = 2N$, for some $N\in \N$. We choose an indexing $$\CA = \{\alpha_1, \ldots, \alpha_{2N}\}$$ so that 
$\alpha_{2i} = \bar{\alpha}_{2i-1}$ for all $i = 1, \ldots, N$. 
Now, let $\Delta$ be the subgraph of $\Gamma$ spanned by those vertices of $\Lambda$ that are not involved in any arrows, and
 $A_\Delta$ the associated RAAG. Choose a free basis $a,b$ of $F_2$, and write $G$ for the direct product of $N$ copies of $F_2 \times F_2$. 
We define a map \[\rho: V(\Lambda) \to A_\Delta \times G\] as follows:

\begin{enumerate}
\item If $v\in V(\Lambda)$ is not involved in any arrows, then we set $\rho(v) = (v,1)$. 
\item Otherwise, we set $\rho(v) = (1, g)$, where $g=(g_1, \ldots, g_N)$, with the $i$-th entry $g_i\in F_2\times F_2$ being equal to $(1,1)$ unless: 
\begin{enumerate}
\item Suppose $v$ is the initial vertex of the arrow  $\alpha_{2i-1}$ (resp. $\alpha_{2i}=\bar{\alpha}_{2i-1}$). In this case, set $g_i=(a,1)$ (resp. $g_i=(b,1)$). 

\item Suppose $v \in \lk(w)$, where $w$ is the initial vertex of the arrow $\alpha_{2i-1}$, and the edge $e=\{v,w\}$ is the edge of the dual arrow $\bar{\alpha}_{2i-1}$
(resp.  $w$ is the initial vertex of the arrow $\alpha_{2i}=\bar{\alpha}_{2i-1}$, and the edge $e=\{v,w\}$ is the edge of the dual arrow $\bar{\alpha}_{2i}={\alpha}_{2i-1}$). In this case, we set $g_i=(a^{-1},a)$ (resp. $g_i=(b^{-1}, b)$). 
\end{enumerate}
\end{enumerate}

\medskip
\begin{center}
 \begin{tikzpicture}[scale=0.6]
 \draw[black]
    (0,0) -- (0,2) 
    (3,0) -- (3,2);
 \draw[->,black]    
    (0,0) -- node[right,below,sloped]{\tiny{$\alpha_{2i}$}}(3,1);
  \draw[->,black]    
     (3,2) -- node[left,above,sloped]{\tiny{$\alpha_{2i-1}$}} (0,1);
      \filldraw (0,0) circle (2pt) node[left=4pt]{$(b,1)$};
        \filldraw (3,0) circle (2pt) node[right=4pt]{$(a^{-1},a)$};
 \filldraw (0,2) circle (2pt) node[left=4pt]{$(b^{-1},b)$};
 \filldraw (3,2) circle (2pt) node[right=4pt]{$(a,1)$};
\end{tikzpicture}
\end{center}
\medskip

   We claim that the map $\rho$ so constructed defines a homomorphism \[\rho: G(\Lambda,\CA)= \PAut(A_\Gamma) \to A_\Delta \times G,\] using a slight abuse of notation. In order to do so, we need to verify that $\rho$ preserves the relations of $\PAut(A_\Gamma)$ given in Theorem \ref{thm:prespaut}. We prove this in two separate claims:


\medskip

\noindent{\bf Claim 1.} Let $v,w\in V(\Delta)$ be linked vertices. Then  \[[v,w] =1\implies [\rho(v) , \rho(w)]=1.\] 

\begin{proof}[Proof of Claim 1.]
 First, suppose $v$ is not involved in any arrows, in which case $\rho(v)= (v,1)$. If $w$ is not involved in any arrows either, then $\rho(w) = (w,1)$. Otherwise, $\rho(w)=(1,g)$ for some $g\in G$. In either case, $[\rho(v) , \rho(w)]=1$ and we are done.  
 
Therefore, we have to consider the case when $v$ and $w$ are both involved in some pair of arrows, where by the above it suffices to assume that they are dual to each other.  Property (D) implies that, up to interchanging $v$ and $w$, we may suppose that $v$ is the initial vertex of $\alpha_{2i-1}$ or  $\alpha_{2i}=\bar{\alpha}_{2i-1}$. Assume, for the sake of concreteness, that we are in the former case, as the other one is totally analogous. In this situation, we have that $\rho(v) = (1, g)$ and $\rho(w) = (1,g')$, where the $i$-th entry of $g$ (resp. $g'$) is $(a,1)$ (resp. $(a^{-1},a)$). In particular, $[\rho(v) , \rho(w)]=1$, as desired. 
\end{proof}

\medskip

\noindent{\bf Claim 2.} For vertices $v,w\in V(\Delta)$, \[[u, vw] =1 \implies [\rho(u), \rho(v)\rho(w)] =1.\] 

\begin{proof}[Proof of Claim 2.]
In this case we know that $v$ and $w$ span an edge $e\in E(\Lambda)$ with $(u,e) \in \CA$, so $ (u,e) = \alpha_{2i-1}$ or $\bar{\alpha}_{2i-1}$ for some $i$. Suppose, again for concreteness, that we are in the former case, so that $u$ is the initial vertex of $\alpha_{2i-1}$. Up to interchanging $v$ and $w$, we may assume that $v$ is the initial vertex of the arrow $\bar{\alpha}_{2i-1}$. In particular, we write $\bar{\alpha}_{2i-1}= (v, \bar{e})$, where $\bar{e}\in E(\Lambda)$ with $u \in \bar e$. Let $z$ be the vertex of $\bar e$ which is not equal to $u$. 
From the definition of $\rho$, we have: 
$$
\begin{aligned}
\rho(u)&=(1,g_u)\\
\rho(v)&=(1,g_v)\\
\rho(w)&=(1,g_w)\\
\rho(z)&=(1,g_z),\\
\end{aligned}
$$
for some $g_u,g_v,g_w,g_z \in G$ whose $i$-th entries are, respectively, equal to $(a,1)$, $(b,1)$, $(b^{-1},b)$, and $(a^{-1},a)$. In particular we see that, when restricted to the $i$-th component of $\rho$, the relation $[\rho(u), \rho(v)\rho(w)] =1$ is preserved. Since the components of $\rho$ correspond precisely to pairs of dual arrows, we obtain  the result. 
This finishes the proof of Claim 2. 
\end{proof}

In the light of Claims 1 and 2, we know that $\rho$ defines a homomorphism that we also denote $\rho$
 \[\rho: \PAut(A_\Gamma) \to  A_{\Delta}\times G,\] where $G$ is the direct product of $N$ copies of $F_2 \times F_2$. It remains to check that the composition $\rho_i$ of $\rho$ with the projection with the $i$-th direct factor of $G$ is surjective, for every $i=1,\ldots, N$. To this end, 
consider the arrows $\alpha_{2i-1}$ and $\alpha_{2i}=\bar{\alpha}_{2i-1}$. By definition, $\alpha_{2i-1}$ corresponds to a relation of the form $[u,vw]=1$; thus, up to renaming $v$ and $w$, $\bar{\alpha}_{2i-1}$  corresponds to the relation $[w,uz]=1$. Hence, in the $i$-th copy of $F_2 \times F_2$, the entries of $\rho(u)$ and $\rho(w)$ are $(a,1)$ and $(b,1)$, and therefore they generate a free group on two generators. Similarly, the entries $\rho(v)$ and $\rho(z)$ are $(b^{-1},b)$ and $(a^{-1}, a)$, respectively. In particular we see that, in this copy of $F_2 \times F_2$, the group generated by $\rho(u)$, $\rho(v)$, $\rho(w)$, and $\rho(z)$ is isomorphic to $F_2 \times F_2$, as desired.
\end{proof}

\begin{example} In the particular case of $\PAut(F_3)$, it is easy to give an explicit description  of the  generating set under the group homomorphism $\rho$ above. This standard generating set has six elements, and is represented by the arrow diagram of  figure \ref{fig:arrows}. In the rest of the example, we will refer to the labelling of that figure. 
 Then one  checks that the images of the vertices of $\Lambda$ under $\rho$ are
$$
\begin{aligned}
\rho(c_{21})&=(a^{-1},a,1,1,a,1)\\
\rho(c_{31})&=(a,1,1,1,a^{-1},a)\\
\rho(c_{32})&=(b,1,a^{-1},a,1,1)\\
\rho(c_{12})&=(b^{-1},b,a,1,1,1)\\
\rho(c_{13})&=(1,1,b,1,b^{-1},b)\\
\rho(c_{23})&=(1,1,b^{-1},b,b,1).\\
\end{aligned}
$$
Observe that $\rho([c_{31},c_{32}])=([a,b],1,1,1,1,1)$ and this element commutes with $\rho(c_{13})$. Therefore
$$\rho([[c_{31},c_{32}],c_{13}])=1$$
so $\rho$ is not injective.
\end{example}

A souped-up version of this example gives:

\begin{proposition}
The homomorphism \[\rho: \PAut(F_n) \to \prod_{i=1}^N F_2 \times F_2\] constructed above is not injective for $n\geq 3$.
\label{lem:nonfaithful}
\end{proposition} 
\begin{proof} We claim that for any $i,j,k$ pairwise distinct
$$\rho([[c_{ij},c_{ik}],c_{ji}])=1.$$
Consider first $\rho(c_{ij})$. Its component on each pair of dual arrows is trivial except in the following cases
\begin{itemize}
\item[(1)$_s$] for $1\leq s\leq n$, $s\neq i,j$: arrow from $c_{sj}$ with dual from $c_{si}$,

\item[(2)$_s$] for $1\leq s\leq n$, $s\neq i,j$: arrow from $c_{ij}$ with dual from $c_{is}$.
\end{itemize}
Now, doing the same for $\rho(c_{ik})$ one sees that for each pair of arrows, one of $\rho(c_{ij})$ or $\rho(c_{i,k})$ is trivial, except in the case (2)$_s$ above for $s=k$. Therefore $\rho([c_{ij},c_{ik}])$ is trivial everywhere except for that precise pair of arrows. As $\rho(c_{ji})$ is trivial on that pair we get the claim.
\end{proof}

\begin{remark}In fact there is no hope to be able to represent $\PAut(F_n)$ faithfully as a subdirect product of a product of free groups of any rank. Indeed, since $\PAut(F_n)$ is of cohomological type $\mathrm{FP}_\infty$ \cite{Brady},  a result of Bridson-Howie-Miller-Short \cite{BHMS} implies that  $\PAut(F_n)$ would be virtually a product of free groups, which is not the case. 
\end{remark}

\section{Normal RAAG subgroups of $\PAut(A_\Gamma)$}
\label{sec:conch}


In this section, we will prove Theorem \ref{thm:main2}.
As stated in the introduction, the map of Theorem \ref{thm:repraag} gives a linear representation of $\PAut(A_\Gamma)$ thus we have a linear quotient of this group. In Theorem \ref{thm:main2} we identify a linear normal subgroup $A_{\hat \Gamma}$ of $\PAut(A_\Gamma)$; more precisely a normal subgroup of $\PAut(A_\Gamma)$ which its itself a RAAG.   
Observe that $\PAut(A_\Gamma)$ contains the subgroup $\Inn(A_\Gamma)\cong A_\Gamma/Z(A_\Gamma)$ of inner automorphisms as a normal subgroup. Moreover, $\Inn(A_\Gamma)$ is well known to be a RAAG, namely the RAAG associated to the graph $\Gamma-\Lambda$ where $\Lambda$ is the set of vertices of $\Gamma$ which are linked to every  vertex other than themselves.
There are other ways to find canonical normal subgroups of $\PAut(A_\Gamma)$ which are normal and RAAGs (see \cite{ChCV1} and Remark \ref{rem:CV} below) but our subgroup $A_{\hat \Gamma}$ is the biggest possible in a certain sense. For example, $\PAut(A_\Gamma)$ is a RAAG itself if and only if  there is no SIL \cite{KP} and this happens if and only if $A=\PAut(A_\Gamma)$.
 
As we will see, the RAAG $A_{\hat \Gamma}$ will be constructed explicitly via a graph $\hat\Gamma$ built from $\Gamma$. 
Before explaining all this, we need some preliminary work.
First, we give the following local characterization of a graph containing no SILs. Given $n\ge 2$ by an {\em $n$-claw} we mean the complete bipartite graph $K_{1,n}$. Given an $n$-claw, its unique vertex of valence $n$ is called the {\em center}, and the other vertices are called the {\em leaves}. We have:


%
%

\
\begin{proposition}\label{SILstar} Let $\Gamma$ be a connected graph. The following conditions are equivalent: 
\begin{enumerate}
\item For each  3-claw embedded as full subgraph, there is a labelling $v_1$,$v_2$,$v_3$ of the leaves so that $v_2$ and $v_3$ are connected in $\Gamma-\st(v_1)$ and  $v_1$ and $v_2$ are connected in $\Gamma-\st(v_3)$.
\item $\Gamma$ has no SIL. 
\end{enumerate}
\end{proposition}

\begin{proof} We first prove (i) $\implies$ (ii). Assume, for contradiction, that $\Gamma$ has a SIL pair $v_1$, $v_2$. We claim that then $v_1$ and $v_2$ are leaves of some 3-claw embedded in $\Gamma$ as full subgraph. 

In order to see this, we first check  that there is a vertex $w$ linked to both $v_1$ and $v_2$. (We remark that this argument appears  in \cite[Lemma 6.2]{GPR}, although we include it here for completeness.) Let $C$ be a shared component of the SIL; recall this means that $C$ is both a connected component of $\Gamma-\st(v_1)$ and of $\Gamma-\st(v_2)$. As $\Gamma$ is connected, there is some path in $\Gamma$ connecting $v_1$ to $C$. But the fact that $C$ is a connected component of $\Gamma-\st(v_1)$ implies that this path has length 2; in particular, there exists some vertex $w$ linked to $v_1$ and to some vertex $v_3$ in $C$. Now, since $C$ is also a connected component of $\Gamma-\st(v_2)$ we deduce that $w$ is linked to $v_2$ too.  At this point, we have our embedded 3-claw with center $w$ and leaves $v_1$, $v_2$ and $v_3$; it remains to check that it is a full subgraph. Observe first that $v_1$ and $v_2$ cannot be linked in $\Gamma$ by the definition of SIL. Neither can $v_1$ be linked to $v_3$ because of the choice of $C$; the same applies to $v_2$ and $v_3$. 
 We claim that (i) cannot hold for the 3-star formed by $v_1,v_2,v_3,w$. Otherwise, we may assume by symmetry that (say) $v_2$ and $v_3$ are connected  in $\Gamma-\st(v_1)$, which is again impossible  by the choice of $C$. 
 
 Conversely, assume that (i) fails. Up to relabeling we may assume that there is a full subgraph of $\Gamma$ which is a 3-claw with leaves  $v_1$, $v_2$ and $v_3$, in such way that $v_1$ and $v_2$ are not connected in $\Gamma-\st(v_3)$ and $v_1$ and $v_3$ are not connected in $\Gamma-\st(v_2)$. Let $C$ be the connected component of $\Gamma-\st(v_3)$ that contains $v_1$. Note that  $v_2\not\in C$, hence $C$ cannot be the dominant component of $\Gamma-\st(v_3)$ with respect to $v_2$. Assume that $C$ is a subordinate component; in other words, $C\subseteq L_2$ where $L_2\subseteq\Gamma-\st(v_2)$ is the dominant component with respect to $v_3$. Then, as $v_1\in C$ and $v_3\in L_2$, we would deduce that $v_1$ and $v_3$ are connected in $\Gamma-\st(v_2)$, which is a contradiction. At this point, $C$ must be a shared component, and thus $v_2$ and $v_3$ form  a SIL pair.  
 \end{proof}

A consequence of the proof of Proposition \ref{SILstar} is that, if $\Gamma$ is connected, then for each SIL pair $v$, $u$ there is a 3-claw with $v$ and $u$ as leaves that fails to satisfy the property in (i). This motivates the following definition.

\begin{definition} Fix a vertex $v\in\Gamma$, and $w$ a vertex linked to $v$. We say that $w$ is {\sl $v$-SIL irrelevant} if for any 3-claw $\Delta\subseteq\Gamma$, with center $w$ and $v$ a leaf, and which is a full subgraph of $\Gamma$, $\Delta$  satisfies condition (i)  of Proposition \ref{SILstar}. In other words, for any full subgraph with form

\medskip
\begin{center}
 \begin{tikzpicture}[scale=0.6]
 \draw[black]
    (0,0) -- (2, 0) -- (3,1);
      \draw[black]
    (2,0) -- (3,-1);
      \filldraw (0,0) circle (2pt) node[below=4pt]{$v$};
        \filldraw (2,0) circle (2pt) node[below=4pt]{$w$};
 \filldraw (3,1) circle (2pt) node[right=4pt]{$v_1$};
 \filldraw (3,-1) circle (2pt) node[right=4pt]{$v_2$};
\end{tikzpicture}
\end{center}
\medskip
we have (possibly after interchanging $v_1$ and $v_2$) that $v$ and $v_1$ are connected in $\Gamma-\st(v_2)$ and at least one of the following holds: either $v_1$ and $v_2$ are connected in $\Gamma-\st(v)$ or $v$ and $v_2$ are connected in $\Gamma-\st(v_1)$. The {\sl slink} of $v$  is now defined as the set of vertices in $\lk(v)$ which are $v$-SIL-irrelevant, that is:
$$\slk(v)=\{w\in\lk(v)\mid w\text{ is $v$-SIL irrelevant}\}.$$ Also, set
$$\sst(v)=\slk(v)\cup\{v\}.$$
\end{definition}

Note that if a vertex $v$ is not part of any SIL-pair, then $\st(v)=\sst(v)$. In fact, the name ``SIL-irrelevant" in the definition above refers to the fact that the removal of $w$ does not create new SIL-type relations (i.e. as in Theorem \ref{thm:prespaut} (iii)) between the partial conjugations by $v$ in the corresponding connected components (this will be clear below).

\begin{example}\label{silirrelevant} Consider the following graphs
\medskip
\begin{center}
 \begin{tikzpicture}[scale=0.6]
 \draw[black]
    (0,1) -- (4, 1)
    (1,0) -- (1,1)
    (3,0) -- (3,1)
     (3,0) -- (4,1);
      \draw[black]
    (9,1) -- (13,1)
    (10,0) -- (13,0)
    (10,0) -- (10,1)
    (11,0) -- (11,1)
    (12,0) -- (12,1)
    (13,0) -- (13,1);
      \filldraw (2,1) circle (2pt) node[above=3pt]{$v$};
 \filldraw (11,1) circle (2pt) node[above=3pt]{$v$};
 \filldraw (1,1) circle (2pt) node[above=3pt]{$w_1$};
 \filldraw (3,1) circle (2pt) node[above=3pt]{$w_2$};
 \filldraw (10,1) circle (2pt) node[above=3pt]{$w_3$};
 \filldraw (12,1) circle (2pt) node[above=3pt]{$w_4$};
 \filldraw (11,0) circle (2pt) node[below=3pt]{$w_5$};

       \filldraw (0,1) circle (2pt);
         \filldraw (1,0) circle (2pt);
           \filldraw (3,0) circle (2pt);
             \filldraw (4,1) circle (2pt);
               \filldraw (9,1) circle (2pt);
                 \filldraw (10,0) circle (2pt);
                   \filldraw (12,0) circle (2pt);
                     \filldraw (13,0) circle (2pt);
\filldraw (13,1) circle (2pt);
       
       \end{tikzpicture}
\end{center}
\medskip
In the first graph, $w_2$ is $v$-SIL irrelevant, but $w_1$ is not. Thus $\sst(v)=\{v,w_2\}$. In the second graph, $w_4$ and $w_5$ are both $v$-SIL irrelevant, while $w_3$ is not. Hence $\sst(v)=\{v,w_4,w_5\}$.
\end{example}

 Our next result relates irrelevant components with the connected components of the complement of a SIL: 
 
\begin{lemma}\label{concomp} Let $\Gamma$ be a connected graph, and suppose $u,v$ form a SIL-pair. Let $L_v$ be the connected component of $\Gamma-\sst(v)$ that contains $u$. Then, for any connected component $D$ of $\Gamma-\st(v)$ which is either a shared or the dominant component of the SIL, we have $D\subseteq L_v$. In other words, whenever  
\begin{itemize}
\item[(i)] either $D$ is also a connected component of $\Gamma-\st(u)$,

\item[(ii)] or $u\in D$,
\end{itemize}
we have $D\subseteq L_v$.
\end{lemma}
\begin{proof} Since $\sst(v)\subseteq\st(v)$ we deduce that if $D$ is the dominant component of $\Gamma-\st(v)$, then $D\subseteq L_v$. Now, assume that $D$ is shared. The proof of Proposition \ref{SILstar} implies that there  exists $w \notin \sst(v)$  which is linked to some vertex  in $C$, and also to both $v$ and $u$. As $u,w\in\Gamma-\sst(v)$ and $C\subseteq\Gamma-\st(v)\subseteq\Gamma-\sst(v)$ we deduce $C\subseteq L_v$ too.  
\end{proof}

\noindent{\bf The graph $\hat \Gamma$.} We are finally in a position to define the (simplicial) graph  $\hat\Gamma$ defining the RAAG $A$ which appears in the statement of Theorem \ref{thm:main2}. The vertices of $\hat \Gamma$ are in bijection with pairs $(L,v)$, where $v$ is a vertex of $\Gamma$ and $L$ is a connected component $\Gamma-\sst(v)$. 
 Two such vertices $(L,v)$ and $(T,u)$ are linked in $\hat\Gamma$ unless $v$ and $u$ are not linked in $\Gamma$, $u\in L$ and $v\in T$.
 
 \medskip
 
 We need one more technical result before proving Theorem \ref{thm:main2}:
 
 \begin{lemma}\label{nontrivial} Let $\Gamma$ be a connected graph, $v$ a vertex of $\Gamma$ and $L\subseteq\Gamma-\sst(v)$ a connected component. Then  $L\nsubseteq\st(v)$.
\end{lemma}
\begin{proof}  Let $w\in L$. First, if $w \notin \st(v)$ we are done. Hence assume that $w \in \st(v)$. 
 Since $w\not\in\sst(v)$, we know that $w$ is not $v$-SIL irrelevant. In other words, there exists a 3-claw embedded as a full subgraph of $\Gamma$, with center $w$ and leaves $v$, $v_2$ and $v_3$, which fails to satisfy condition of (i) in Proposition \ref{SILstar}. In particular, $v_2\not\in\st(v)$ thus $v_2\not\in\sst(v)$ also but as $v_2$ and $w\in L$ are linked, $v_2\in L\cap(\Gamma-\st(v))$, as desired.
\end{proof}

We are finally in a position to prove Theorem \ref{thm:main2}: 

\begin{proof}[Proof of Theorem \ref{thm:main2}]

We first define $\varphi$ on the vertices of $\hat\Gamma$ by the formula
$$\varphi(L,v)=c_{L,v}$$
where recall that $c_{L,v}$ is the partial conjugation of $L$ by $v$. 
We first check that $\varphi$ is well-defined. To this end, it suffices to see that $\varphi$ preserves the standard relators of $A_{\hat\Gamma}$. Let 
$(L,v)$ and $(T,u)$ be two vertices. If either $u=v$ or $u\in \lk_\Gamma(v)$, then the partial conjugations $c_{L,v}$ and $c_{T,u}$ commute. 
So we may assume that $v$ and $u$ are not linked in $\Gamma$ and, say, $v\not\in T$. Note that $T\cap(\Gamma - \st(u))=T-(T\cap\st(v))$ is a union of connected components of $\Gamma-\st(v)$, denoted $D_1,\ldots,D_r$, so that
$$c_{T,u}=c_{D_1,u}\ldots c_{D_r,u}$$
 and that none of them is the dominant component of the SIL as $v\not\in D_1\cup\ldots\cup D_r$. Now we distinguish two cases: 
\begin{itemize}
\item[(i)] Suppose first that $L$ is the connected component of $\Gamma - \sst(v)$ that contains $u$. Then, Proposition \ref{concomp} yields that $L$ contains the dominant and all the shared components of $\Gamma-\st(v)$ with respect to $u$. Now, Theorem \ref{presA} implies that 
$[c_{L,v},c_{D_i,u}]=1$, for $i=1,\ldots,r$. Hence $[c_{L,v},c_{T,u}]=1$, as desired.

\item[(ii)] Otherwise, $L$ must be a union of shared components, and we are done again by Theorem \ref{presA}.
\end{itemize}

At this point, we know that the map $\varphi$ is well-defined. Next, let $\Lambda\subseteq\Gamma$ be the full subgraph whose vertices are linked to every vertex of $\Gamma$ other than themselves. By a result of Servatius \cite{Servatius}, the subgroup $A_\Lambda$ is precisely the center of $A_\Gamma$. Moreover, if $\Gamma_0=\Gamma-\Lambda$ then 
$A_{\Gamma_0}\cong A_\Gamma/A_\Lambda=\Inn(A_\Gamma).$
Thus we have a map 
$$\iota:A_{\Gamma_0}\to A_{\hat\Gamma}$$
given by the rule $$v\mapsto\prod\{(L,v)\mid L\text{ connected component of }\Gamma-\sst(v)\},$$
which is obviously well-defined. To see that its image is normal in $A_{\hat\Gamma}$ we proceed as follows. Let  $u,v$ be vertices of $\Gamma$, and $T$ any connected component of $\gamma-\sst(u)$.
If $v\not\in T$, then $$\iota(v)^{(T,u)}=\iota(v).$$
Otherwise, if $u\in L$,
$$(L,v)^{(T,u)}=(L,v)^{\iota(v)},$$
while if $u\notin L$
$$(L,v)^{(T,u)}=(L,v).$$
In particular,
$$\iota(v)^{(T,u)}=\iota(v)^{\iota(u)},$$ as desired. 

Now, the maps $\iota$ and $\varphi$ fit in a commutative diagram similar to the one in \cite[Page 9]{Charney}:

\begin{equation*}
\xymatrix{
&{1}\ar[r]&{A_{\Gamma_0}}\ar[r]^{\iota} \ar[d]_{\cong}
&{A_{\bar\Gamma}} \ar[r]\ar[d]_{\varphi} &{Q}\ar[r]\ar[d]_{\bar\varphi} &1 \\
&{1}\ar[r]&{\Inn(A_\Gamma)} \ar[r]
&{\im(\varphi)}\ar[r]&{\hat Q}\ar[r] &1 
}
\end{equation*}
where the injectivity of $\iota$ follows from the commutativity of the diagram, $Q$ and $\hat{Q}$ are the corresponding quotient groups and $\bar\varphi$ is the induced map.

We now check that  $\im(\varphi)$ is normal in $\PAut(A_\Gamma)$. Let $u,v$ be  vertices of $\Gamma$, $L$ a connected component of $\Gamma-\sst(v)$, and $C$ a connected component of $\Gamma-\st(u)$. The same argument used above to show that $\varphi$ is well-defined shows that $[c_{L,v},c_{C,u}]=1$, except  possibly in the case when $C$ is the dominant component, i.e., when $v\in C$. In this  case, let $T$ be the connected component of $\Gamma-\sst(u)$ that contains $C$. Then $T-(T\cap\st(u))=C\cup C_1\cup\ldots\cup C_s$  for some (non-dominant) components $C_1,\ldots,C_s$ and we get $$[c_{L,v},c_{C_i,u}]=1$$
for $i=1,\ldots,s$. As a consequence,
$$c_{L,v}^{c_{C,u}}=c_{L,v}^{c_{C,u}\prod_{i=1}^sc_{C_i,u}}=c_{L,v}^{c_{T,u}}\in\im(\varphi).$$

Finally, we claim that the induced map $\bar\varphi:Q \to \hat Q$ is injective. We follow the strategy of \cite[Theorem 3.6]{Charney}, and first show that $Q$ is abelian. To this end, let $(L,v)$ and $(T,u)$ be two vertices of $\hat\Gamma$, which are necessarily linked unless $v\in T$. In this case,  if $T_1,\ldots,T_r$ are the remaining connected components of $\Gamma-\sst(u)$, we have $[(L,v),(T_i,u)]=1$ for $i=1,\ldots,r$, and thus
$$[(L,v),(T,u)]=[(L,v),(T,u)\prod_{i=1}^r(T_i,u)]\in A_{\Gamma_0},$$ as desired. 
In fact, it is easy to give an explicit presentation of $Q$, through which one sees that $Q$ is the free abelian group generated by the cosets of the $(L,v)$'s, after removing exactly one of them for each $v$ (in an arbitrary manner). 

 To see that $\bar\varphi$ is injective, observe that if we had a linear combination $\bar{z}$ of elements such that $\bar\varphi(\bar{z})=1$ then any preimage $z\in A_{\hat{\Gamma}}$ would be such that $\varphi(z)\in\Inn(A_\Gamma)$. Let $(L,v)$ be some vertex  whose coset appears in $\bar{z}$. As $L\not\subseteq\st(v)$ by Lemma \ref{nontrivial}, $\varphi(z)$ conjugates any $w\in L$ by (at least) $v$. However, it does not conjugate any $u\in T_v$ and, using again that  $T_v\not\subseteq\st(v)$,
we deduce that $\varphi(z)$ cannot be inner. 
\end{proof}

\begin{remark}
Observe that if $\Gamma$ has no SIL, then $\sst(v)=\st(v)$ for any $v$ thus $\hat\Gamma$ has a vertex for each connected component of $\Gamma-\st(v)$ and the group $A_{\hat\Gamma}$ is precisely $\PAut(A_\Gamma)$, as shown already in \cite{Charney}.
\end{remark}

\begin{remark} In \cite{ChCV1}, in the case when $\Gamma$ is connected and has no triangles, the authors construct a map from a finite index subgroup of $\Out(A_\Gamma)$ to a product of groups of the form $\Out(F)$, with $F$ free. The kernel $K$ turns out to be abelian, and is generated by the cosets in $\Out(A_\Gamma)$ of those partial conjugations which are of the form  $c_{L,v}$ where $L$ is a connected component of $\Gamma-\{v\}$.  This construction was subsequently generalized  \cite{ChV2}  to  arbitrary graphs $\Gamma$.

 In general, if we consider the subgroup $N$ of $\PAut(A_\Gamma)$ generated by partial conjugations $c_{L,v}$, with $L$ a connected component of $\Gamma - \{v\}$, one can proceed along similar lines as above to show  that $N$ is a RAAG  containing $\Inn(A_\Gamma)$ and that it is normal in $\PAut(A_\Gamma)$ and the corresponding quotient is abelian. Moreover, $N$ embeds in $A_{\hat\Gamma}$, but in general  as a proper subgroup. For example, this is what happens for the graphs of Example \ref{silirrelevant}.
\label{rem:CV}
\end{remark}

We finish this section with a few words about a possible obstruction to the linearity of $\PAut(A_\Gamma)$ in the case when $\Gamma$ has a SIL. Consider the group $G$ defined as
$$G=(F_2\times F_2)*_t$$
where the stable group is the diagonal group $\Delta(F_2\times F_2)$ and $t$ acts as the identity on $\Delta(F_2\times F_2)$, i.e., $(g,g)^t=(g,g)$. $G$ admits the following explicit presentation:
$$G=\langle x,a,y,b,t\mid [x,y]=[x,b]=[a,y]=[a,b]=[xy,t]=[ab,t]=1\rangle.$$
 Prassidis and Metaftsis  have conjectured that this group is linear. In fact, if $G$ is not linear then the next lemma would imply that $\PAut(A_\Gamma)$ cannot be linear either, provided $\Gamma$ has a SIL:

 \begin{lemma} If $\Gamma$ has a SIL, then $\PAut(A_\Gamma)$ has a subgroup isomorphic to $(F_2\times F_2)*_t$ where $(g,g)^t=(g,g)$.
\end{lemma}
\begin{proof} Observe that we may express the defining relators of $G$ as
 $$\begin{aligned}
G=&\langle x,a,y,b,t\mid [x,y]=[x,b]=[a,y]=[a,b]=[xy,t]=[ab,t]=1\rangle\\
&\langle x,a,y,b,t\mid a^y=a^b=a, x^y=x^b=x, t^y=t^{x^{-1}}, t^b=t^{a^{-1}}\rangle,\\
\end{aligned}$$  
and that this means that $G$ is the semidirect product of the (normal) subgroup $K$ generated by $x$, $a$ and $t$, which is free of rank 3, and a free group of rank 2 generated by $y$ and $b$, where the action is given by partial conjugations. 

Now, assume that $\Gamma$ has a SIL pair $v$ and $u$, and let $C$ be a shared component. Choose some vertex $w\in C$ and consider the inner automorphisms $c_u$, $c_v$ and $c_w$ that conjugate by $u$, $v$ and $w$ respectively. Consider also the partial conjugations $c_{C,v^{-1}}$ and $c_{C,u^{-1}}$, and let $H$ be the subgroup of $\PAut(A_\Gamma)$ generated by these five elements. We claim that the map $f:G\to H$ induced by
$$
\begin{aligned}
x\mapsto c_u,\\
a\mapsto c_v,\\
t\mapsto c_w,\\
y\mapsto c_{C,u^{-1}},\\
b\mapsto c_{C,v^{-1}}\\
\end{aligned}$$
is an isomorphism. To see this, observe that as the vertices $v$, $u$ and $w$ are pairwise not linked, the subgroup of $A_\Gamma$ that they generate is free of rank 3, and  thus the same holds for the conjugations $c_u$, $c_v$, $c_w$. Thus the restriction of $f$ yields an isomorphism from $K$ to the subgroup generated by $c_u$, $c_v$ and $c_w$. The same happens for the restriction of $f$ to the subgroup of $G$ generated by $y$ and $b$ and the subgroup of $H$ generated by $c_{C,u^{-1}}$ and $c_{C,v^{-1}}$. Taking into account the structure of $G$ we deduce that $f$ is in fact an isomorphism.
\end{proof}

 \begin{remark}
It is not difficult to construct explicit linear representations of $G$.
Let $A,B,D,M$ be matrices so that $[A,D]=1=[B,D]$ and $A,B$ generate a free group (we might choose $D=I$). Represent the generators of $G$ as follows:
$$x\mapsto\text{diag}(A,I)$$
$$y\mapsto\text{diag}(I,A)$$
$$a\mapsto\text{diag}(B,I)$$
$$b\mapsto\text{diag}(I,B)$$
$$t\mapsto M\otimes D$$
where $\otimes$ is the Kronecker product. Note that $xy\mapsto\text{diag}(A,A)=I\otimes A$ which commutes with $M\otimes D$.
\end{remark}
In light of the above remark, a natural question is: 

\begin{question} Is there is some choice of $A$, $B$, $D$ and $M$ so that this representation is faithful?
\end{question}

\end{document}